\documentclass[11pt]{amsart}

\usepackage{amsthm}
\usepackage{amsmath}
\usepackage{amssymb}
\usepackage{color}
\usepackage{graphics, graphicx}

\newtheorem{thm}{Theorem}[section]
\newtheorem{theorem}[thm]{Theorem}

\newtheorem{cor}[thm]{Corollary}

\newtheorem{lemma}[thm]{Lemma}
\newtheorem{proposition}[thm]{Proposition}

\theoremstyle{remark}
\newtheorem{remark}[thm]{Remark}

\newcommand{\RR}{\mathbb R}

\newcommand{\ip}[2]{\left\langle #1,#2\right\rangle}

\newcommand{\nm}[1]{\|{#1}\|}
\newcommand{\fr}{\mathcal{F}}

\begin{document}

\title{Optimal properties of the canonical tight probabilistic frame}
\author[D. Cheng and K. A. Okoudjou]{Desai Cheng and Kasso A.~Okoudjou}
\address{Department of Mathematics, University
of Missouri, Columbia, MO 65211-4100}
\email{chengdesai@yahoo.com}
\address{Department of Mathematics and Norbert Wiener Center, University
of Maryland, College Park, MD 20742}
\email{kasso@math.umd.edu}
%\thanks{This work  was partially supported by a grant from the Simons Foundation $\# 319197$, and ARO grant W911NF1610008.}

\date{\today}


\begin{abstract} A probabilistic frame is a Borel probability measure with finite second moment whose support spans  $\RR^d$. A Parseval probabilistic frame is one for which the associated matrix of second moment is the identity matrix in $\RR^d$. Each probabilistic frame is canonically associated to a Parseval probabilistic frame. In this paper, we show that this canonical Parseval probabilistic frame is the closest Parseval probabilistic frame to a given probabilistic frame  in the $2-$Wasserstein distance. Our proof is based on two main ingredients. On the one hand, we show that a probabilistic frame can be  approximated in the $2-$Wasserstein metric  with (compactly supported) finite frames whose bounds can be controlled. On the other hand we establish some fine continuity properties of the function that maps a probabilistic frame to its canonical Parseval probabilistic frame.
Our results generalize  similar ones for finite frames and their associated Parseval frames.
\end{abstract}

\maketitle

%\tableofcontents

\section{Introduction}\label{sec1}

The notion of probabilistic frames was first  introduced in \cite{EhlerRTF} in the setting of probability measures on the unit sphere, and was later generalized to probability measures  on $\RR^d$ in \cite{EhlOko2012}. In essence, this theory is a generalization of the theory of finite frames which has seen a wealth of activities in recent year, \cite{CasKut2013, chri2003, KovChe1, KovChe2, OkoudjouFiniteFrame}.

\subsection{Review of finite frame theory}\label{subsec1.1}
Before we give the definition and some elementary properties of probabilistic frames, we recall that a set $\Phi=\{\varphi_i\}_{i=1}^N \subset \RR^d$ is a frame for $\RR^d$ if and only if there exist $0<A\leq B<\infty$ such that $$A\|x\|^2\leq \sum_{i=1}^N\ip{x}{\varphi_i}^2\leq B\|x\|^2\qquad \forall\, x\in \RR^2.$$
The  frame $\Phi$ is a \emph{tight frame} if we can choose $A = B$.  Furthermore, if $A=B=1$, $\Phi$ is called a Parseval frame. In the sequel the set of frames for $\RR^d$ with $N$ vectors will be denoted by $\fr(N,d)$, and simply  $\fr$ when the context is clear. The subset of frames with frame bounds $0<A\leq B<\infty$ will be denoted $\fr_{A, B}(N,d),$ or simply $\fr_{A,B}$.  We equip the set  $\fr(N, d)$  with the metric
\begin{equation}\label{frame-metric}
d(\Phi, \Psi)= \sqrt{\sum_{i=1}^{N}\|\varphi_i - \psi_i\|^2}=\sqrt{\sum_{i=1}^d\|R_i-P_i\|^2}
\end{equation} where $\Phi=\{\varphi_i\}_{i=1}^{N}, \Psi=\{\psi_i\}_{i=1}^{N }) \in \fr(M,d),$   $\{R_i\}_{i=1}^d, \{P_i\}_{i=1}^d \subset \RR^N$ denote the rows of $\Phi$, and those of $\Psi$, respectively.

Let $\Phi=\{\varphi_i\}_{i=1}^N$ be a frame for $\RR^d$. Throughout the paper we shall abuse notation and denote the  \emph{synthesis matrix} of the frame by   $\Phi$,  the $d\times N$ whose $i^{th}$ column is $\varphi_i.$ The matrix $$S:=S_{\Phi}=\Phi \Phi^{T}=\sum_{i=1}^N \ip{\cdot}{\varphi_i}\varphi_i$$ is the \emph{frame matrix}. It is known that $\Phi=\{\varphi_i\}_{i=1}^N$ is a frame for $\RR^d$ if and only if $S$ is a positive definite matrix. Moreover, the smallest eigenvalue of $S$ is the optimal lower frame bound, while its largest eigenvalue is the optimal upper frame bound. $\Phi$ is a tight frame if and only if $S$ is a multiple of the $d\times d$ identity matrix. In particular, $\Phi$ is a Parseval frame if and only if $S=I$.

If $\Phi$ is a frame, then $S$ is positive definite and thus invertible. Consequently,  $$\Phi^{\dag}=\{\varphi_i^\dag\}_{i=1}^N= \{S^{-1/2}\varphi_i\}_{i=1}^N$$ is a Parseval frame, leading to following reconstruction formula:

$$x=\sum_{i=1}^N\ip{x}{\varphi_i^{\dag}}\varphi_i=\sum_{i=1}^N\ip{x}{\varphi_i} \varphi_i^{\dag}\, \forall\, x\in \RR^d.$$

In addition,  $\Phi^\dag$ is the unique Parseval frame which solves the following problem \cite[Theorem 3.1]{CasKut07}:

\begin{equation}\label{min-clos-pf}
\min\{d(\Phi, \Psi)^2=\sum_{i=1}^N\|\varphi_i-\psi_i\|^2: \Psi=\{\psi_i\}_{i=1}^N\subset \RR^d, \, \, {\textrm Parseval \, frame}\}.
\end{equation}
To be specific,

\begin{theorem}\label{uniqueness_prob_1}\cite[Theorem 3.1]{CasKut07}

If $\Phi=\{\varphi_i\}_{i=1}^N$ is a frame for $\RR^d$,  then $\Phi^{\dag}=\{\varphi_i^\dag\}_{i=1}^N= \{S^{-1/2}\varphi_i\}_{i=1}^N$ is the unique solution to~\eqref{min-clos-pf}.
\end{theorem}

In Section~\ref{sec2}, and for the sake of completeness, we give a new and simple proof of this result and we refer to  \cite{Balan99, BodCas10, CahillCasazza13} for related results.

\subsection{Probabilistic frames}\label{subsec1.2}
The main goal of this paper is to characterize the minimizers of an optimal problem analog of ~\eqref{min-clos-pf} for probabilistic frames.  To motivate the definition of a probabilistic frame, we note that given a  frame $\Phi = \{\varphi_i\}_{i=1}^N \subset \RR^{d}$, then the discrete probability  measure $$\mu_{\Phi}=\tfrac{1}{N}\sum_{k=1}^N\delta_{\varphi_k}$$ has the property that its support ($\{\varphi_k\}_{k=1}^N$) spans $\RR^d$ and that it has finite second moment, i.e., $$\int_{\RR^d}\|x\|^2d\mu_{\Phi}(x)=\tfrac{1}{N}\sum_{k=1}^N\|\varphi_k\|^2< \infty.$$  The  probability measure $\mu_\Phi$ is an example of a probabilistic frame that was introduced in \cite{EhlerRTF, EhlOko2012}.

More specifically, a Borel probability measure $\mu$   is a \emph{probabilistic frame} if there exist $0<A\leq B < \infty$ such that for all $x\in\RR^d$ we have
\begin{equation}\label{pfineq}
 A\|x\|^2 \leq \int_{\RR^{d}} |\langle x,y\rangle |^2 d\mu (y) \leq B\|x\|^2.
 \end{equation}
The constants $A$ and $B$ are called \emph{lower and upper probabilistic frame bounds}, respectively.
When $A=B,$ $\mu$ is called a \emph{tight probabilistic frame}. In particular, when $A=B=1$, $\mu$ is called a \emph{Parseval probabilistic frame}.

A special class of probabilistic frames that will be considered in the sequel consists of discrete measures $\mu_{\Phi, w}=  \sum_{i=1}^{N}w_i\delta_{\varphi_i}$ where  $\Phi=\{\varphi_i\}_{i=1}^N\subset \RR^d$, and $w=\{w_i\}_{i=1}^N\subset [0, \infty)$ is a set of  weights such that $\sum_{i=1}^Nw_i=1$. A probability measure such as $\mu_{\Phi, w}$ will be termed  \emph{finite probabilistic frame}, if and only if  it is a probabilistic frame for $\RR^d$. When the context is clear we will simply write $\mu$ for $\mu_{\Phi, w}$.  We shall also identify a finite probabilistic frame $\mu_{\Phi, w}$ with the frame $\Phi_w=\{\sqrt{w_i}\varphi_i\}_{i=1}^{N}$, as both have the same frame bounds. We refer to the surveys \cite{EhlOko2013, KOPrecondPF2016} for an overview of the theory of probabilistic frames.

We shall prove an analog of Theorem~\ref{uniqueness_prob_1} by endowing the set of probabilistic frames with  the Wasserstein metric.
Let $\mathcal{P}:=\mathcal{P}(\mathcal{B},\RR^d)$ denote the collection of probability measures on $\RR^d$ with respect to the Borel $\sigma$-algebra $\mathcal{B}$.
Let $$
\mathcal{P}_{2}:= \mathcal{P}_{2}(\RR^{d})=\bigg\{ \mu \in \mathcal{P}: M_{2}^{2}(\mu):=\int_{\RR^{d}}\nm{x}{}^{2}d\mu(x) < \infty\bigg\}$$
be the set of all probability measures with finite second moments.  For $\mu, \nu \in \mathcal{P}_2$, let
$\Gamma(\mu, \nu)$ be the set of all Borel probability measures  $\gamma$ on $\RR^d \times \RR^d$ whose marginals are $\mu$ and $\nu$, respectively, i.e., $\gamma(A\times \RR^d)=\mu(A)$ and $\gamma(\RR^d \times B) = \nu(B)$ for all Borel subset $A, B$ in $\RR^d$. The space
$\mathcal{P}_{2}$ is equipped with the $2$-\emph{Wasserstein metric}  given by
\begin{equation}\label{wmetric}
W_{2}^{2}(\mu, \nu):=\min\bigg\{\int_{\RR^d \times \RR^d}\nm{x-y}{}^{2}d\gamma(x, y), \gamma \in \Gamma(\mu, \nu)\bigg\}.
\end{equation}
 The minimum defined by~\eqref{wmetric} is achieved at a  measure $\gamma_0 \in \Gamma(\mu, \nu)$, that is:
$$W_{2}^{2}(\mu, \nu)=\int_{\RR^d \times \RR^d}\nm{x-y}{}^{2}d\gamma_0(x, y). $$
  We refer to \cite[Chapter 7]{AGS2005}, and  \cite[Chapter 6]{Villani2009} for more details on the Wasserstein spaces.

\subsection{Our contributions}\label{subsec1.3} The investigation of probabilistic frames is still at its initial stage. For example, in \cite{WCKO16} the authors introduced the notion of transport duals and used the setting of the Wasserstein metric to investigate the properties of such probabilistic frames. In particular, this setting offers the flexibility to find (non-discrete) probabilistic frames which are duals to a given probabilistic frame.  Transport duals are the probabilistic analogues of alternate duals in frame theory \cite{chri2003, Larsen}. The main contribution of this paper (Theorem~\ref{maintheorem}) is to investigate the properties of the canonical Parseval probabilistic frame associated to a given  probabilistic frame, see Section~\ref{sec2} for definitions. To prove this result we approximate a given probabilistic frame with one that is compactly supported and whose frame bounds are controlled in a precise way (Theorem~\ref{density-dpf}).  In the process of proving our main result, we prove a number of results that are of interest on their own right. For example, in Section~\ref{sec2} we  establish a number of new results about the canonical Parseval frame $\Phi^{\dag}$  associated to a frame $\Phi$.

\section{Optimal Parseval probabilistic frames}\label{sec2}
Before proving our main result in Section~\ref{subsec2.3}, we revisit  the canonical Parseval frame $\Phi^{\dag}$ associated to a given frame $\Phi=\{\varphi_k\}_{k=1}^N\subset \RR^d$. In particular, Section~\ref{subsec2.1} considers the continuity  properties of the map $F(\Phi)=\Phi^\dag$. In Section~\ref{subsec2.2} we show how a probabilistic frame can be approximated in the $2$-Wasserstein metric  by a sequence of finite frames whose bounds are controlled by those of the initial probabilistic frame. While such approximation for  probability measures in the $2$-Wasserstein metric is well known \cite[Theorem 6.18]{Villani2009}, our key contribution here is the control of the frame bounds of the approximating sequence.

\subsection{Continuity properties of the canonical Parseval frame}\label{subsec2.1}

In this section we revisited the canonical Parseval frame $\Phi^{\dag}$ associated to a given frame $\Phi=\{\varphi_k\}_{k=1}^N\subset \RR^d$.
First, we give a new and elementary proof of Theorem~\ref{uniqueness_prob_1}.

\begin{proof}{Proof of Theorem~\ref{uniqueness_prob_1}}
We first note that a frame $\Psi\subset \RR^d$ is Parseval if the rows of its synthesis matrix are orthonormal. Furthermore, $\Psi \subset \RR^d$ is a Parseval frame if and only if $U\Psi$ is a Parseval frame for any $d\times d$ orthogonal matrix $U$.

Now let   $\Phi=\{\varphi_i\}_{i=1}^N$ be a frame for $\RR^d$.  Write $S=\Phi\Phi^T= UDU^T$ for some orthogonal matrix $U$. Observe that $U^T\Phi$ is the matrix of $\Phi$ written with respect to the orthonormal basis given by the rows of $U^T$. In addition, the rows of $U^T\Phi$ are pairwise  orthogonal. Let $\Psi=\{\psi_i\}_{i=1}^N\subset \RR^d$ be any Parseval frame, then

$$d^2(\Phi, \Psi)=d^2(U^T\Phi, U^T\Psi)=\sum_{i=1}^d\|R_i-P_i\|^2,$$ where $\{R_i\}_{i=1}^d\subset \RR^N$ and $\{P_i\}_{i=1}^d\subset \RR^N$  denote respectively  the rows of $U^T\Phi$  and $U^T\Psi$.
Consequently, finding $$\min\{d(\Phi, \Psi)^2=\sum_{i=1}^N\|\varphi_i-\psi_i\|^2: \Psi=\{\psi_i\}_{i=1}^N\subset \RR^d, \, \, {\textrm Parseval \, frame}\}$$ is equivalent to finding
$$\min \{\sum_{i=1}^d\|R_i-P_i\|^2: \{P_i\}_{i=1}^N\subset \RR^N, \, \, {\textrm orthonormal \, set }\}$$  where $\{R_i\}_{i=1}^d$ form an orthogonal set of vectors in $\RR^N$.

But $\Phi^\dag=\{S^{-1/2}\varphi_i\}_{i=1}^N$  is a Parseval frame, so its rows form an orthonormal set in $\RR^N$. Consequently,  $\Phi^\dag$  is  a solution to~\eqref{min-clos-pf}.  The uniqueness follows by observing that the (unique) closest orthonormal set to a given orthogonal vectors $\{u_i\}_{i=1}^d\subset \RR^N$ is $\{\tfrac{u_i}{\|u_i\|}\}_{i=1}^d.$

Consequently, $$\min\{d(\Phi, \Psi)^2=\sum_{i=1}^N\|\varphi_i-\psi_i\|^2: \Psi=\{\psi_i\}_{i=1}^N\subset \RR^d, \, \, {\textrm Parseval \, frame}\}=\sum_{k=1}^d(1-\lambda_k^{-1/2})^2 $$ where $\{\lambda_k\}_{k=1}^d\subset (0, \infty)$ are the eigenvalues of $S=\Phi\Phi^T$.

\end{proof}

In the remaining part of  section we study the continuity properties of the functions that maps a given frame to its canonical Parseval frame. This map  $$F: \fr(N, d) \rightarrow \fr(N,d)$$ given by
\begin{equation}\label{func-F}
F(\Phi)=F(\{\varphi_i\}_{i=1}^{N}) = S_{\Phi}^{-1/2}(\{\varphi_i\}_{i=1}^{N})=\{S_{\Phi}^{-1/2}\varphi_i\}_{i=1}^N.
\end{equation} In fact, our results show that for $0< A\leq B$, $F$ is uniformly continuous on $\fr_{A, B}$, the set of frames with frame bounds between $A$ and $B$. More specifically,

\begin{theorem}\label{cont-F}
Let $0<A\leq B <\infty$, and  $\delta > 0$ be given. Then there  exists  $\epsilon > 0$ such that given any frame  $\Phi=\{\varphi_i\}_{i=1}^N$,  with frame bounds between $A$ and $B$, and $N:=N_{\Phi}\geq 2$, for any  frame $\Psi=\{\psi_i\}_{i=1}^N$ such that $d(\Phi, \Psi)< \epsilon$ we have $d(F(\Phi), F(\Psi))< \delta.$
\end{theorem}

Before proving this theorem, we establish a number of preliminary results and make the following remark that will be used in the sequel.

\begin{remark}\label{rowsframe}
 Let $\Phi=\{\varphi_i\}_{i=1}^{N}\in \fr(N, d)$ be a frame. Then,  $S=\Phi \Phi^{T}= ODO^{T}$ where $O$ is a $d\times d$ orthogonal matrix and $D$ is a positive definite diagonal matrix. Fix the orthonormal basis of $\RR^d$ whose columns form the matrix $O$ and write each frame vector $\varphi_i$ in this basis. The synthesis matrix of the frame $\Phi$ in the basis $O$ is $$[\Phi]_{O}=O^{T}\Phi.$$ Let $\{R_i\}_{i=1}^d$ be the rows of $[\Phi]_{O}$. We shall refer to $\{R_i\}_{i=1}^d$ as simply the rows of  $\Phi$.
\end{remark}

\begin{lemma}\label{rows-prop} Let $\Phi=\{\varphi_i\}_{i=1}^{N} \in \fr(N, d)$. Denote by $\{R_i\}_{i=1}^{d}$ the rows of $\Phi$  as described by Remark~\ref{rowsframe}.  Let $\epsilon>0$ and $ \Psi=\{\psi_i\}_{i=1}^{N} \in \fr(N, d)$ be such that $d(\Phi, \Psi) <  \epsilon $. Denote by $\{P_i\}_{i=1}^d$ the rows of $\Psi$ when written in the orthonormal basis $O$.
Then
\begin{enumerate}
\item[(a)]  $\big|\|\ R_i\|-\|P_i\|\big|< \epsilon.$ Furthermore, $\sqrt{A}-\epsilon <\|P_i\|<\sqrt{B}+\epsilon$ for each $i=1, 2, \hdots, d.$
\item[(b)] $$d(\Phi,F(\Phi)) \geq \sqrt{\sum_{i=1}^{d}\|R_i - \dfrac{R_i}
{\|R_i\|}\|^2}.$$
\item[(c)] For each $i\in \{1, 2\, \hdots, d\}$ we have
$$\bigg\|\dfrac{P_i}{\|P_i\|} - \dfrac{R_i}{\|R_i\|}\bigg\| <  \dfrac{2\epsilon}{\sqrt{A}}.$$
\item[(d)] For each $i\in \{1, 2\, \hdots, d\}$ we have $$0\leq \|P_i-\tfrac{R_i}{\|R_i\|}\|^2-\|P_i-\tfrac{P_i}{\|P_i\|}\|^2\leq \tfrac{4\epsilon}{\sqrt{A}}c+\tfrac{4\epsilon^2}{A},$$ where $c=max(1-\sqrt{A}+\epsilon, \sqrt{B} + \epsilon - 1)$.
\end{enumerate}
\end{lemma}

\begin{proof}
\begin{enumerate}
\item[(a)] This is trivial so we omit it.
\item[(b)] This follows immediately from the fact that the rows of a Parseval frame are an orthonormal set when written with respect to any orthonormal basis and $\dfrac{R_i}{\|R_i\|}$ is the closest unit norm vector  to $R_i$.
\item[(c)]
Since, $d(\Phi, \Psi)<\epsilon,$ we know that $\big|\|P_i\| - \|R_i\|\big| < \epsilon$. Hence
$$ \bigg\|\dfrac{P_i}{\|P_i\|}\cdot\|R_i\| - R_i\bigg\| \leq \bigg\|\dfrac{P_i}{\|P_i\|}\cdot\|R_i\| - P_i\bigg\| +\|P_i-R_i\|= \big|\|P_i\| - \|R_i\|\big|+\|P_i-R_i\|< 2\epsilon.$$
The result follows by recalling that $\|R_i\|\geq \sqrt{A}$.
\item[(d)]  It is clear that $\|P_i -\dfrac{P_i}{\|P_i\|}\| =|\|P_i\|-1| \leq max(1-\sqrt{A}+\epsilon, \sqrt{B} + \epsilon - 1) = c$.
By part (c) we know that  $\|\dfrac{P_i}{\|P_i\|} - \dfrac{R_i}{\|R_i\|}\| <  \dfrac{2\epsilon}{\sqrt{A}}$. Using the fact hat $\dfrac{P_i}{\|P_i\|}$ is the closest unit norm vector to $P_i$,  we see that
$$\|P_i - \dfrac{P_i}{\|P_i\|}\| \leq \|P_i - \dfrac{R_i}{\|R_i\|}\| \leq \|P_i - \dfrac{P_i}{\|P_i\|}\| + \dfrac{2\epsilon}{\sqrt{A}}.$$
The result follows by squaring the last inequality.
\end{enumerate}
\end{proof}

Finally, we have the following technical lemma, that contains the key argument in the proof of Theorem~\ref{cont-F}.

\begin{lemma}\label{contradiction-unifcont} Given $0<A\leq B<0$, fix $\Phi=\{\varphi_i\}_{i=1}^N\in \fr_{A,B}$. Let $\epsilon, \delta>0$ be such that $\dfrac{\delta}{\sqrt{d}} - \dfrac{2\epsilon}{\sqrt{A}}>0$ and $\sqrt{A} - \epsilon>0$.  Let $\Psi=\{\psi_i\}_{i=1}^N$ be such that  $d(\Phi, \Psi) <  \epsilon$, and $d(S_{\Phi}^{-1/2}\Phi, S^{-1/2}_{\Psi}\Psi) =d(\Phi^{\dag}, \Psi^{\dag})> \delta$.  Then,
$$\sum_{i=1}^{d}(\|P_i - R_{i}'\|^2 - \|P_i - \dfrac{P_i}{\|P_{i}\|}\|^2) \geq min(Cd'^2,C^2), $$ where
$d' = \dfrac{\delta}{\sqrt{d}} - \dfrac{2\epsilon}{\sqrt{A}}$, $C = \min(\sqrt{A} - \epsilon, 1)$, and $\{R_i'\}_{i=1}^d\subset \RR^d$ is the set of the rows of $S_{\Psi}^{-1/2}\Psi$.
\end{lemma}

\begin{proof} 
We first show  that there exists $k$ then $$\|P_k- R_{k}'\|^2 - \|P_k - \dfrac{P_k}{\|P_{k}\|}\|^2 \geq min(\|R'_{k} - \frac{P_k}{\|P_k\|}\|^2 \cdot min(\|P_k\|,1),\|P_k\|^2).$$

Since $d(S_{\Phi}^{-1/2} \Phi,S_{\Psi}^{-1/2} \Psi) \geq \delta$,  then $\|\dfrac{R_k}{\|R_k\|} - R_{k}'\| \geq \dfrac{\delta}{\sqrt{d}}$ for some $k$. By Lemma~\ref{rows-prop} we know that   $\|\dfrac{P_k}{\|P_k\|} - \dfrac{R_k}{\|R_k\|}\| <  \dfrac{2\epsilon}{\sqrt{A}}$. It follows from the triangle inequality that
 $$\|\dfrac{P_k}{\|P_k\|} - R_{k}'\| \geq \dfrac{\delta}{\sqrt{d}} - \dfrac{2\epsilon}{\sqrt{A}} = d'.$$

Suppose that $C =\min(\sqrt{A} - \epsilon, 1)= 1$, or equivalently,  $\sqrt{A} - \epsilon \geq 1$. Hence, by Lemma~\ref{rows-prop} we have $\|P_i\|\geq 1$ for each for all $i$.

Since the angle $\widehat{R_k'\tfrac{P_k}{\|P_k\|}P_i}> \pi/2$, it follows that   $$\|P_k - R_{k}'\|^2 > \|P_k - \dfrac{P_k}{\|P_k\|}\|^2 + \|\dfrac{P_k}{\|P_k\|} - R_{k}'\|^2.$$ But since,  $\|\dfrac{P_k}{\|P_k\|} - R_{k}'\| \geq \dfrac{\delta}{\sqrt{d}} - \dfrac{2\epsilon}{\sqrt{A}}$,  we conclude that
$$  \|P_k- R_{k}'\|^2 - \|P_k - \dfrac{P_k}{\|P_k\|}\|^2>   \|\dfrac{P_k}{\|P_k\|} - R_{k}'\|^2 \geq d'^2 = Cd'^2$$ and we are done.

Assume now $C=\sqrt{A}-\epsilon < 1$  and $\|P_k\|+\eta \leq 1$, where $\eta$ is defined in Figure~\ref{fig:figure1}.

\begin{figure}[htbp]
\begin{center}
\includegraphics[scale=.35]{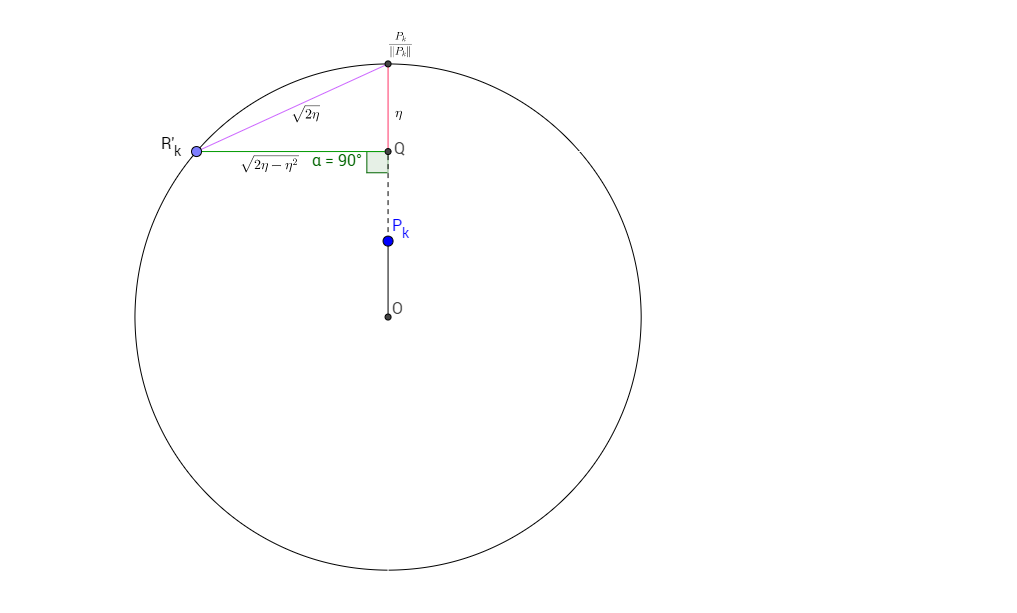}
\end{center}
\caption{  $Q$ is  the orthogonal projection of $R'_k$ onto  $P_k$, and $\eta= \|Q - \frac{P_k}{\|P_k\|}\|$. }\label{fig:figure1}
\end{figure}

Then,
\begin{align*}
\bigg\|P_k - R_{k}'\|^2 - \|P_k - \dfrac{P_k}{\|P_{k}\|}\bigg\|^2 &=
(1 - ( \|P_k\|+ \eta))^2 + 2\eta - \eta^2 - (1-\|P_k\|)^2 \\
&= 2\eta \|P_k\| \\
&= \bigg\|\dfrac{P_k}{\|P_k\|} - R_{k}'\bigg\|^{2}\|P_k\|.
\end{align*} The  the conclusion follows  from  $\|\dfrac{P_k}{\|P_k\|} - R_{k}'\|^{2} \geq d'^2$.

Now assume $\|P_k\|+\eta > 1$ and $\eta \leq 1$, where $\eta$ is defined in Figure~\ref{fig:figure2}.

\begin{figure}[htbp]
\begin{center}
\includegraphics[scale=.35]{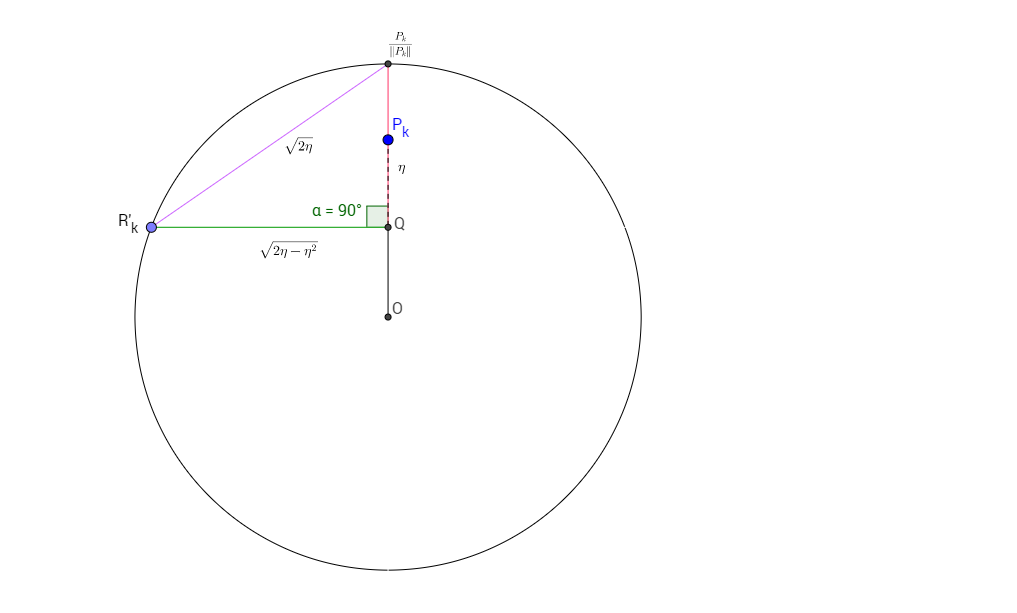}
\end{center}
\caption{$Q$ is  the orthogonal projection of $R'_k$ onto  $P_k$, and $\eta= \|Q - \frac{P_k}{\|P_k\|}\|$.}\label{fig:figure2}
\end{figure}

$$\bigg\|P_k - R_{k}'\|^2 - \|P_k - \dfrac{P_k}{\|P_{k}\|}\bigg\|^2 = ((\|P_k\| + \eta) - 1)^2 + 2\eta - \eta^2 - (1-\|P_k\|)^2 = 2\eta \|P_k\|$$ and the rest of the proof is similar to the one given above.

If $\eta > 1$ where where $\eta$ is defined in Figure~\ref{fig:figure3}, then the angle $\angle P_{k}0R'_{k} > \frac{\pi}{2}$ hence $\|P_k - R_{k}'\|^2 > \|P_k\|^2 + 1$. We know $\|P_k - \dfrac{P_k}{\|P_{k}\|}\|^2 \leq 1$ hence $$\|P_k - R_{k}'\|^2 - \|P_k - \dfrac{P_k}{\|P_{k}\|}\|^2 > \|P_k\|^2 \geq C^2$$

\begin{figure}[htbp]
\begin{center}
\includegraphics[scale=.35]{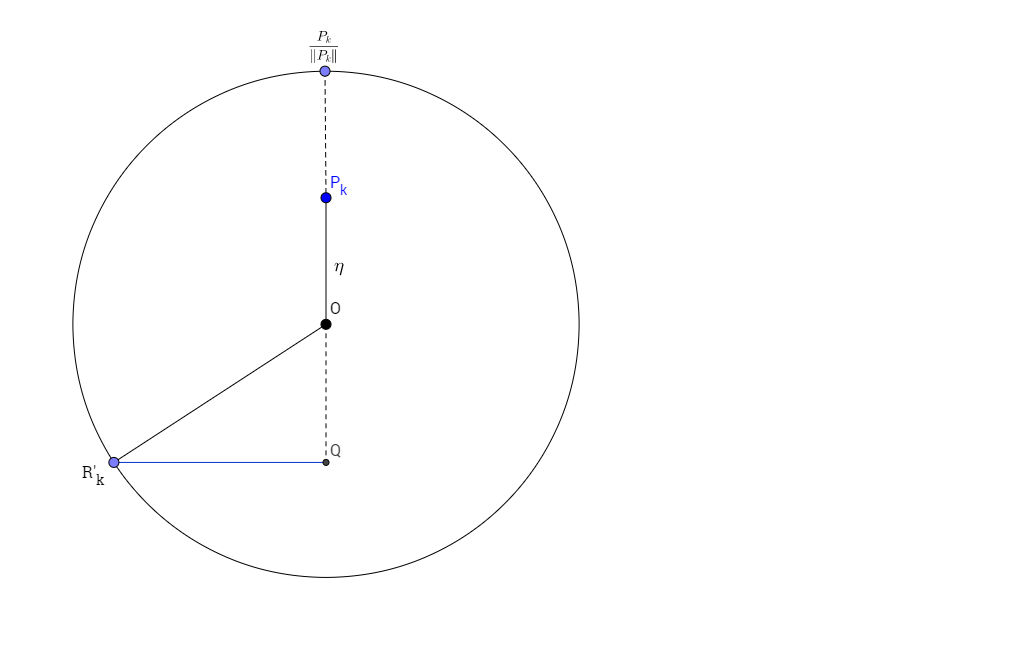}
\end{center}
\caption{$Q$ is  the orthogonal projection of $R'_k$ onto  $P_k$, and $\eta= \|Q - \frac{P_k}{\|P_k\|}\|$.}\label{fig:figure3}
\end{figure}

\end{proof}
We are now ready to prove Theorem~\ref{cont-F}.

\begin{proof}[Proof of Theorem~\ref{cont-F}]
Assume by way of contradiction that there exists $\delta > 0$ such that for all $\epsilon > 0$ there exist $\Phi_{\epsilon}=\{\varphi_{i,\epsilon}\}_{i=1}^{N_{\epsilon}}, \in \fr_{A, B}.$ and
$\Psi_\epsilon=\{\psi_{i,\epsilon}\}_{i=1}^{N_{\epsilon}}$

 such that $$d(\Phi_{\epsilon}, \Psi_{\epsilon})< \epsilon$$ and $$d(S_{\Phi_{\epsilon}}^{-1/2}\Phi_{\epsilon}, S_{\Psi_{\epsilon}}^{-1/2}\Psi_{\epsilon}) > \delta.$$  Furthermore, choose  $\epsilon$ small enough so that $\dfrac{\delta}{\sqrt{d}} - \dfrac{2\epsilon}{\sqrt{A}}>0$ and $\sqrt{A} - \epsilon>0$ and
  $$\sum_{i=1}^{d}(\|P_i - \dfrac{R_i}{\|R_i\|}\|^2 - \|P_i - \dfrac{P_i}{\|P_i\|}\|^2) <  min(Cd'^2,C^2)$$ where $C$ and $d'^2$ are as in Lemma~\ref{contradiction-unifcont}(such $\epsilon$ exists by Lemma~\ref{rows-prop}).

  Hence $$\sum_{i=1}^{d}(\|P_i - \dfrac{R_i}{\|R_i\|}\|^2 - \|P_i - \dfrac{P_i}{\|P_i\|}\|^2) <  \sum_{i=1}^{d}(\|P_i - R_{i}'\|^2 - \|P_i - \dfrac{P_i}{\|P_{i}\|}\|^2)$$ Consequently,  $\sum_{i=1}^{d}\|P_i - \dfrac{R_i}{\|R_i\|}\|^2 <  \sum_{i=1}^{d}\|P_i - R_{i}'\|^2$ contradicting that $R_{i}'$ are the rows of the closest Parseval frame to $\Psi_{\epsilon}=\{\psi_{i,\epsilon}\}_{i=1}^{N_{\epsilon}}$.
\end{proof}

\subsection{Approximation of probabilistic frames in the $2-$ Wasserstein metric}\label{subsec2.2}
In this section we prove some of the technical results needed to establish our main result.
The key idea is that a probabilistic frame $\mu$ with frame bounds $A, B$ can be approximated in the Wasserstein metric by a finite probabilistic frame whose bounds are arbitrarily close to $A, B$.   We prove this statement in Proposition~\ref{density-dpf} and point out that it is a refinement of a well-known result, e.g.,  \cite[Theorem 6.18]{Villani2009}. But first, we prove a few new results about finite probabilistic frames that are of interest in their own right.  In particular, Lemma~\ref{frame_modi} will be a very useful technical tool that we shall often use. It shows that given a finite frame we may replace any frame vector by a finite number of new vectors so as to leave unchanged the frame operator. More specifically,

\begin{lemma}\label{frame_modi}
Given a frame $\Phi=\{\varphi_{i}\}_{i=1}^N$ with frame operator $S_\Phi$. Fix $i\in \{1, 2, \hdots, N\}$ and consider the new set of vectors $$\Phi_i=\{\varphi_k\}_{k=1, k\neq i}^N \cup\{a_{j}\varphi_i\}_{j=1}^{p}= \{\varphi_k'\}_{k=1}^{N+p-1} $$ where $\sum_{j=1}^{p}a_{j}^{2} = 1$. Then, $\Phi_i\in \fr(N+p-1, d)$, that is, $\Phi_i$ is  a frame for $\RR^d$ and its frame operator is $S_{\Phi}$. Furthermore, $$\sum_{k=1}^{N}\|\varphi_k- \varphi_k^{\dag}\|^2 = \sum_{k=1}^{N+p-1}\|\varphi_k' -  \varphi_k^{' \dag}   \|^2$$ where $\varphi_k^{\dag}=S^{-1/2}\varphi_k$ and $\varphi_k^{'\dag}=S^{-1/2}\varphi'_k$
\end{lemma}

\begin{proof} It is easy to see that for each $x\in \RR^d$ we have
$$\sum_{k=1}^{N}|\ip{x}{\varphi_k}|^2 = \sum_{k=1}^{i-1}|\ip{x}{\varphi_k}|^2 + \sum_{j=1}^{p}a_j^2|\ip{x}{\varphi_i}|^2 + \sum_{k=i+1}^{N}|\ip{x}{\varphi_k}|^2 .$$
\end{proof}

We now use Lemma~\ref{frame_modi} and Theorem~\ref{uniqueness_prob_1} to find the closest Parseval frame to a finite probabilistic frame in the $2$-Wasserstein metric.

\begin{proposition}\label{dists-to-parseval}
Let  $\mu_{\Phi, w}$ be a finite probabilistic frame with bounds $A$ and $B$, where $\Phi=\{\varphi_{i}\}_{i=1}^N \subset \RR^{d}$ and  $w=\{w_i\}_{i=1}^N \subset [0, \infty)$. Then the closest finite Parseval probabilistic frame to $\Phi$ is $\Phi^\dag=\{S^{-1/2}\varphi_i\}_{i=1}^N$ and it satisfies

$$W_{2}(\mu_{\Phi, w}, \mu_{\Phi^{\dag}, w})=\sqrt{\sum_{i=1}^{N}w_{i}\|\varphi_i - \tilde{\varphi}_i\|^2} \leq  \sqrt{d\, \max((\sqrt{A} - 1)^2,(\sqrt{B} - 1)^2)}$$
 where $\tilde{\varphi}_i=S^{-1/2}\varphi_i$.
\end{proposition}

\begin{proof}
We first prove that $W_{2}^2(\mu_{\Phi, w}, \mu_{\Phi^{\dag}, w}) \leq  d\, \max((\sqrt{A} - 1)^2,(\sqrt{B} - 1)^2).$

Let  $\Phi_w=\{\sqrt{w_i}\varphi_{i}\}_{i=1}^N$. Let $S=\Phi_w\Phi^T_w=ODO^T$ be the frame operator of $\Phi_w$. Consider the columns of $O$ as an orthonormal basis for $\RR^d$. Writing the vectors $\sqrt{w}_k\varphi_k$ with respect to this basis leads to  $\Phi_w'=O^T\Phi_w$ where  \[\Phi_w = \left( \begin{array}{ccc}
| & ... & | \\
\sqrt{w_1}\varphi_1 & ... & \sqrt{w_n}\varphi_m \\
| & ... & | \end{array} \right)\]
Let $\{P_{k, w}\}_{k=1}^d$ and $\{R_{k, w}\}_{k=1}^d$  respectively denote the rows of $\Phi_w'$ and  $\Phi_w$. Notice that $$\sqrt{A}\leq \|P_{k,w}\|\leq \sqrt{B}, \quad \forall\, k=1, 2, \hdots, d.$$ It is easily seen that

$$\min_{u\in \RR^d, \|u\|=1}\|P_{k,w}-u\|^2=\|P_{k,w}-\tfrac{P_{k,w}}{\|P_{k,w}\|}\|^2=|\|P_{k,w}\|-1|^2\leq \max((\sqrt{A} - 1)^2,(\sqrt{B} - 1)^2).$$ But by construction, $\ip{P_{k,w}}{P_{\ell,w}}=0$ for $k\neq \ell$, and $ \tfrac{P_{k,w}}{\|P_{k,w}\|}=\lambda_k^{-1/2}P_{k,w}$ where $\lambda_k$ is the $k^{th}$ eigenvalue of $S$. Consequently, $\{\lambda_k^{-1/2}P_{k,w}\}_{k=1}^d$ represents the rows of the canonical tight frame $S^{-1/2}\Phi_w$ written in the orthonormal basis $O$. Therefore,

$$d(\Phi_w, S^{-1/2}\Phi_w)^2=\sum_{k=1}^d\|P_{k,w}-\lambda_k^{-1/2}P_{k,w}\|^2\leq d \max((\sqrt{A} - 1)^2,(\sqrt{B} - 1)^2).$$

Clearly,
$$W_{2}^2(\mu_{\Phi, w}, \mu_{S^{-1/2}\Phi, w}) \leq \sum_{i=1}^{N}w_{i}\|\varphi_{i} - S^{-1/2}\varphi_{i}\|^2=d(\Phi_w, S^{-1/2}\Phi_w)^2\leq d \max((\sqrt{A} - 1)^2,(\sqrt{B} - 1)^2) .$$
Suppose there exists a finite probabilistic Parseval frame $\mu_{\Psi, v}$ where $\Psi =\{\psi_i\}_{i=1}^{M}\subset \RR^d$, $v=\{v_i\}_{i=1}^M\subset [0, \infty)$ such that $$W_{2}^2(\mu_{\Phi, w}, \mu_{\Psi, v}) < \sum_{i=1}^{N}w_{i}\|\varphi_{i} - S^{-1/2}\varphi_{i}\|^2.$$
Let $\gamma \in \Gamma (\mu_{\Phi, w}, \mu_{\Psi, v})$ be such that $$W_2^2(\mu_{\Phi, w}, \mu_{\Psi, v}) =\iint_{\RR^{2d}} \|x-y\|^2d\gamma(x,y).$$
Note that $\gamma$ is a discrete measure with $\gamma(x,y)=\sum_{i, j}w'_{i,j}\delta_{\varphi_{i}}(x)\delta_{\psi_{i}}(y)$ with
 $\sum_{j}w'_{i,j} = w_i$ and $\sum_{i}w'_{i,j} = v_j$.

Furthermore, by assumption $$W_{2}^2(\mu_{\Phi, w}, \mu_{\Psi, v})=\sum_{i,j}w'_{i,j}\|\varphi_i - \psi_j\|^2 <  \sum_{i=1}^{N}w_{i}\|\varphi_{i} - S^{-1/2}\varphi_{i}\|^2.$$
Notice since $\sum_{i}w'_{i,j} = v_j$ the frame $\Psi'=\{\sqrt{w'_{i,j}}\psi_j\}_{i,j}$ is a Parseval frame. Since $\sum_{j}w'_{i,j} = w_i$, it easy to see that $\sum_{j} \tfrac{w'_{i,j}}{w_i}   =1$.  We now  use Lemma~\ref{frame_modi}. For each $i,$ replace  $\sqrt{w_i}\varphi_i$ with $\{\sqrt{w_{i,j}'}\varphi_i\}_{j}$. This results in a frame $\Phi'=\{\sqrt{w_{i,j}'}\varphi_i\}_{i, j}$. Consequently, $d(\Phi', \Psi')=d(\Phi_w, \Psi_v)<d(\Phi_w, \Phi_w^{\dag})$ where $\Psi_v$ is a Parseval frame. This is a contradiction.
\end{proof}

The next result is one of our key technical results. It allows us to approximate a probabilistic frame in the $2$-Wasserstein metric with a compactly supported finite probabilistic frame whose bounds are controlled by those of the original probabilistic frame.

\begin{theorem}\label{density-dpf}
Let $\mu$ be a  probabilistic frame with frame bounds $A$ and $B$, and $\epsilon > 0$. Then,  there exists a finite probabilistic $\mu_{\Phi}$ with frame bounds $A', B'$ such that $A'\geq A-\epsilon$, $B'\leq B+\epsilon$ and $$\|\mu - \mu_{\Phi}\|_{W_2}<\epsilon.$$
\end{theorem}

To establish this result we first prove the following two Lemmas.

\begin{lemma}\label{prop-2}
Let $\mu$ be a probabilistic frame with frame bound $A$ and $B$. Given $\epsilon>0$, there exists a probabilistic frames $\nu$ with compact support and frame bounds $A', B'$ such that
\begin{enumerate}
\item[(a)] $W_2^2(\mu, \nu)< \epsilon$,
\item[(b)] $A'\geq A-\epsilon$, and $B'=B$.
\end{enumerate}
\end{lemma}

\begin{proof}
\begin{enumerate}
\item[(a)] Let $\mu$ be a probabilistic frame with frame bound $A$ and $B$. Given $\epsilon>0$, there exists  $R_1>0$ such that $$\int_{\RR^{d}\setminus B(0, R_1)} \|x\|^{2}d\mu(x) < \epsilon.$$

Let $\nu$ be the measure defined for each Borel set $A\subset \RR^d$ by  $$\nu(A) = \mu(A\bigcap B(0,R_1)+ \mu(\RR^d\setminus B(0,R_1))\delta_0.$$  Clearly,  $\nu$ is a probabilistic measure with compact support.

 We consider the marginal $\gamma$ of $\mu$ and $\nu$ defined for each Borel sets $A, B\subset \RR^d$   by
 \begin{equation*}
\gamma(A\times B)= \left\{ \begin{array} {r@{\quad {\textrm if} \quad}l}
\mu(A\bigcap B(0,R_1) \bigcap B) + \mu(A\bigcap B^{c}(0,R_1) & 0\in B\\
\mu(A\bigcap B(0,R_1) \bigcap B) & 0\not\in B \end{array}\right.
\end{equation*}
Since $\nu $ is supported in $B(0,R_1)$
 \begin{align*}
 \iint_{\RR^{2d}}\|x-y\|^{2}d\gamma(x,y) &= \iint_{\RR^d \times B(0,R_1)}\|x-y\|^{2}d\gamma(x,y)\\
 &=\iint_{B(0,R_1) \times B(0,R_1)}\|x-y\|^{2}d\gamma(x,y)  \\
 &+ \iint_{B^{c}(0,R_1)\times B(0,R_1)}\|x-y\|^{2}d\gamma(x,y).
 \end{align*}
 However,  we know
$$\int_{B(0,R_1)\times B(0,R_1)}\|x-y\|^{2}d\gamma(x,y)=0$$ since, when restricted to $B(0,R_1)\times B(0,R_1)$, $\gamma$ is supported only on the diagonal where $\|x-y\| = 0.$ Moreover,
\begin{align*}
\int_{B^{c}(0,R_1)\times B(0,R_1)}\|x-y\|^{2}d\gamma(x,y) &=\iint_{B^{c}(0,R_1) \times B(0,R_1)\setminus \{0\}}\|x-y\|^{2}d\gamma(x,y)\\
& + \iint_{B^{c}(0,R_1)\times \{0\}}\|x-y\|^{2}d\gamma(x,y)\\
&=0+\iint_{B^{c}(0,R_1)\times \{0\}}\|x-y\|^{2}d\gamma(x,y)\\
&<\epsilon.
\end{align*}
Therefore, $W_2^2(\mu, \nu)<\epsilon.$

 \item[(b)] The upper bound $B$ is obtained trivially as $\nu$ is $\mu$ restricted to $B(0,R_1)$.

For  $x\in \RR^d$ we have $\int|\left\langle  x,y\right\rangle|^{2}d\nu(y) = \int_{B(0,R_1)}|\left\langle  x,y\right\rangle|^{2}d\mu(y).$ From the fact that $\int_{\RR^d \setminus B(0,R_1)}\|x\|^2d\mu(x) \leq \epsilon$ it follows that $$\int_{\RR^d \setminus B(0,R_1)}|\left\langle  x,y\right\rangle|^{2}d\mu(y)\leq \|x\|^2\epsilon.$$
\end{enumerate}
\end{proof}

Suppose that $\mu$ is a probabilistic frame supported in a ball $B(0, R)$. Let $r>0$ and consider $Q=[0, r)^d$. Choose points $\{c_k\}_{k=1}^M\subset \RR^d$ with $c_1=0$ such that $B(0, R)=\cup_{k=0}^M Q_k$ where $Q_k=c_k+Q$. Observe that $Q_k\cap Q_\ell=\emptyset$ whenever $k\neq \ell$. Let $\mu_{1,Q}=\sum_{k=1}^M\mu(Q_k)\delta_{c_k}$.

Next partition each cube $Q_k$ uniformly into cube of size $r/2$ and construct the probability measure $\mu_{2,Q}$ as above. Iterate this process to construct a sequence of probability measures $\mu_{n, Q}$.

\begin{lemma}\label{prop-3} Let $\mu$ be a probabilistic frame with bounds $A$ and $B$, which supported in a ball $B(0, R)$. For  $r>0$  let $\{\mu_{n,Q}\}_{n=1}^{\infty}$ be a sequence of probability measures as constructed above. Then, $$lim_{n\to \infty}W_2(\mu, \mu_{n,Q})=0.$$ Furthermore, there exists $N$ such that for all $n\geq N$, $\mu_{n, Q}$ is a finite probabilistic
frame whose bounds are arbitrarily close to those of $\mu$.
\end{lemma}

\begin{proof}
Let $d=\max_{x\in Q_k}\|x-c_k\|$.  Given, $x\in Q_k$, $x=c_k + a_k$, where $\|a_k\|\leq d$.

For any $x\in \RR^d$,
\begin{align*}
\bigg|\int_{B(0, R)}\ip{x}{y}^2d\mu(y)-\sum_{k=1}^M\ip{x}{c_k}^2\mu(Q_k)\bigg|&= \bigg|\sum_{k=1}^M\int_{Q_k}\ip{x}{y}^2 d\mu(y)-\sum_{k=1}^M\ip{x}{c_k}^2\mu(Q_k)\bigg|\\
&=\bigg|\sum_{k=1}^M\int_{Q_k}(\ip{x}{y}^2-\ip{x}{c_k}^2)d\mu(y)\bigg|\\
&\leq \sum_{k=1}^M\int_{Q_k}\big|\ip{x}{y}^2-\ip{x}{c_k}^2\big|d\mu(y)\\
&=\sum_{k=1}^M\int_{Q_k}|\ip{x}{c_k+a_k}^2-\ip{x}{c_k}^2|d\mu(y)\\
&=\sum_{k=1}^M\int_{Q_k}|\ip{x}{a_k}^2+2\ip{x}{c_k}\ip{x}{a_k}|d\mu(y)\\
&\leq \|x\|^2\sum_{k=1}^M\mu(Q_k)(\|a_k\|^2+2\|c_k\|\|a_k\|)\\
&\leq (d^2+2d(R+d))\|x\|^2.
\end{align*}

Note that by the iterative construction of $\mu_{n, Q}$ we get that for each $x\in \RR^d$ $$\bigg|\int_{\RR^{d}} \ip{x}{y}^2d\mu(y)-\int_{\RR^{d}} \ip{x}{y}^2d\mu_{n, Q}(y)\bigg|\leq (d_{n}^2+2d_{n}(R+d_{n}))\|x\|^2$$ where $\lim_{n\to \infty}d_n=0$. It follows that given $\epsilon>0$, we can find $N>1$ such that for all $n\geq N$,

$$\int_{\RR^{d}} \ip{x}{y}^2d\mu_{n, Q}(y)>\int_{\RR^{d}} \ip{x}{y}^2d\mu(y)-\epsilon\|x\|^2>\|x\|^2(A-\epsilon)$$ which concludes that $\mu_{n, Q}$ is a a finite probabilistic frame whose lower bound is at least $A-\epsilon$. Furthermore, $$\int_{\RR^{d}} \ip{x}{y}^2d\mu_{n, Q}(y)<\int_{\RR^{d}} \ip{x}{y}^2d\mu(y)+\epsilon\|x\|^2\leq \|x\|^2(B+\epsilon)$$ which implies that the upper frame bound  $\mu_{n, Q}$ is at most $B+\epsilon$.

Next, fix $n\geq N$ and let $\gamma_n(x,y)$ be the measure on $\RR^d\times \RR^d$ be defined  for any Borel sets  $ A, B \subset \RR^d$  by:

$$\gamma_n(A\times B)=\sum_{k: c_k\in B}\mu(A\cap Q_k)=\sum_{k=1}^M\mu_{|_{Q_k}}\times \delta_{c_{k}}(A\times B)$$ where $A, B$  $c_k$ denoting the centers of the cubes $Q_{k}$. It is easy to see  that  $\gamma_n \in \Gamma(\mu, \mu_{n, Q})$
and so
\begin{align*}
W_2^2(\mu, \mu_{n, Q})&\leq \iint\|x-y\|^2d\gamma_n(x, y)\\
&=\sum_{k=1}^{M}\iint \|x-y\|^2d(\mu_{|_{Q_{k}}}\times \delta_{c_{k}})(x,y)\\
&=\sum_{k=1}^M \int_{Q_k}\|x-c_k\|^2d\mu(x)\\
&\leq \sum_{k=1}^M\mu(Q_k)\int_{Q_k}d_n^2d\mu(x)\\
&\leq d_n^2
\end{align*}
and the result follows from the fact that $\lim_{n\to \infty}d_n=0$.

\end{proof}

We can now give a proof  of  Theorem~\ref{density-dpf}.
\begin{proof}[Proof of Theorem~\ref{density-dpf}]
Let $\mu$ be a  probabilistic frame with frame bounds $A$ and $B$, and $\epsilon > 0$.  By Lemma~\ref{prop-2} let  $\nu$ be a compactly supported probabilistic frame with  frame bounds between $A-\epsilon/2$ and $B$ and such that $W_2(\mu, \nu)< \epsilon/2$.

By Lemma~\ref{prop-3}  we know there exists a finite probabilistic frame $\mu_{\Phi, w}$ whose frame bounds are within $\epsilon/2$ of that of $\nu$ and such that $W_2(\nu, \mu_{\Phi, w})< \epsilon/2 $.  Consequently, $W_2(\mu, \mu_{\Phi, w})<\epsilon$ which concludes the proof.
\end{proof}

\begin{cor}\label{appro-prob}
Let $\mu$ be a probabilistic Parseval frame and $\epsilon > 0.$ Then, there exists a finite Parseval probabilistic frame $\mu_{\Phi, w}$ with $$W_{2}(\mu, \mu_{\Phi, w})< \epsilon.$$
\end{cor}
\begin{proof}
This follows from Proposition~\ref{dists-to-parseval} and Theorem~\ref{density-dpf}.
\end{proof}
\begin{remark}
Since the set of finite Parseval frames is dense in the set of all Parseval frames in the Wasserstein metric, by Proposition 2.6 since there is no finite Parseval frame closer to $\Phi$  than $\Phi^\dag=\{S^{-1/2}\varphi_i\}_{i=1}^N$, there are no Parseval frame closer to $\Phi$  than $\Phi^\dag$.
\end{remark}

\subsection{The closest  Parseval frame  in the $2-$Wasserstein distance}\label{subsec2.3}

In this section we prove and state of our main result, Theorem~\ref{maintheorem}. We recall that if $\mu$ is a probabilistic frame for $\RR^d$, then its  probabilistic frame operator (equivalently, the matrix of second moments associated to $\mu$)

$$S_\mu:\RR^d\rightarrow \RR^d,\qquad S_\mu (x) = \int_{\RR^d} \ip{ x}{y} y d\mu(y)$$ is positive definite, and thus $S_\mu^{-1/2}$ exists. We define the push-forward of $\mu$ through $S_{\mu}^{-1/2}$  by $$\mu^{\dagger} (B) =\mu(S^{1/2} B)$$ for each Borel set in $\RR^d$. Alternatively, if $f$ is a continuous bounded function on $\RR^d$, $$\int_{\RR^d} f(y)d \mu^{\dag}(y) = \int_{\RR^d} f(S_{\mu}^{-1/2}y) d\mu(y).$$

It then follows that
$$
x=S_{\mu}^{-1/2}S_{\mu} S_{\mu}^{-1/2}(x)=   \int_{\RR^d}   \ip{S_{\mu}^{-1/2} x}{y} \, S_{\mu}^{-1/2}y \, d\mu(y)=   \int_{\RR^d} \ip{x}{ y}\, y \, d\mu^{\dagger} (y) $$ implying that $\mu^{\dag}$ is a Parseval probabilistic frame \cite{EhlOko2013, KOPrecondPF2016}. In particular, $S_{\mu^{\dag}}=I$ where $I$ is the identity matrix on $\RR^d$. As was the case with the canonical Parseval frame $\Phi^{\dag}$ of a given frame $\Phi$, $\mu^{\dag}$ is the (unique) closest  Parseval probabilistic frame to $\mu$.

\begin{theorem}\label{maintheorem}
Let $\mu$ be a probabilistic frame on $\RR^d$ with probabilistic frame operator $S_\mu$. Then $\mu^{\dag}$ is the (unique) closest probabilistic Parseval frame to $\mu$ in the $2-$Wasserstein metric, that is
\begin{equation}\label{optimum}
\mu^{\dag}=\textrm{arg} \min  W_{2}^{2}(\mu, \nu)
\end{equation}
where $\nu $ ranges over all Parseval probabilistic frames.
\end{theorem}

Before proving this theorem, we need to establish a few preliminary results. We start by extending Theorem~\ref{cont-F} to finite probabilistic frames in the Wasserstein metric. In particular, this extension allows use to deal with finite probabilistic frames of different cardinalities.

\begin{theorem}\label{continuityFPF}
Let $0<A\leq B <\infty$, and  $\delta > 0$ be given. Then there  exists  $\epsilon > 0$ such that given any finite probabilistic frame $\mu_{\Phi, w}=\sum_{i=1}^Nw_i\delta_{\varphi_i}$ with   frame bounds between $A$ and $B$, $N:=N_{\Phi}\geq 2$,  $\Phi=\{\varphi_i\}_{i=1}^N\subset \RR^d$, and weights $w=\{w_i\}_{i=1}^N \subset [0, \infty)$, for any  finite probabilistic frame $\mu_{\Psi, \eta}=\sum_{i=1}^M\eta_i\delta_{\psi_i},$ $M:=M_{\Psi}\geq 2$, where $\Psi=\{\psi_i\}_{i=1}^M\subset \RR^d$, and weights $\eta=\{\eta_i\}_{i=1}^N \subset [0, \infty)$  if  $W_2(\mu_{\Phi, w}, \mu_{\Psi, \eta})< \epsilon,$  then we have $$W_2(F(\mu_{\Phi, w}), F(\mu_{\Psi, \eta})) < \delta.$$
\end{theorem}

\begin{proof} Fix $\delta>0$.  By Theorem~\ref{cont-F} we know that there exists $\epsilon$ such that given a frame $X=\{x_i\}_{i=1}^M$ ($M\geq 2$ is arbitrary) with frame bounds between $A$ and $B$, and  $Y=\{y_{i}\}_{i=1}^M$   is a frame such that $$d(X, Y)=\sqrt{\sum_{i=1}^{M}\|x_{i} - y_{i}\|^2} <  \epsilon$$ then $$d(F(X), F(Y))=d(S^{-1/2}_{X}X, S^{-1/2}_{Y}Y) <  \delta.$$

Let $\mu_{\Phi, w}=\sum_{i=1}^N w_i\delta_{\varphi_i}$ be a finite probabilistic frame with   frame bounds between $A$ and $B$, $N \geq 2$,  $\Phi=\{\varphi_i\}_{i=1}^N\subset \RR^d$, and weights $w=\{w_i\}_{i=1}^N \subset [0, \infty)$. Then by Theorem~\ref{dists-to-parseval}, $\mu_{\Phi^{\dag}, w}$ where $\Phi^{\dag}=\{S^{-1/2}_{\Phi}\varphi_i\}_{i=1}^N$ is the closest Parseval frame to $\mu_{\Phi, w}$.

Let $\mu_{\Psi, v}$ where $\Psi=\{\psi_i\}_{i=1}^M$, $M\geq 2$ such that $W_2(\mu_{\Phi, w}, \mu_{\Psi, \eta})< \epsilon$. Choose $\gamma \in \Gamma(\mu_{\Phi, w}, \mu_{\Psi, v})$  such that
$$W_2(\mu_{\Phi, w}, \mu_{\Psi, \eta})^2=\iint_{\RR^d \times \RR^d}\|x-y\|^2d\gamma(x,y)< \epsilon^2.$$ Identify $\gamma$ with $\{w_{i,j}\}_{i, j=1}^{N, M}$. Then,

$$W_2(\mu_{\Phi, w}, \mu_{\Psi, \eta})^2=\iint_{\RR^d \times \RR^d}\|x-y\|^2d\gamma(x,y)=\sum_{i=1}^M\sum_{j=1}^Nw_{i,j}\|\varphi_i-\psi_j\|^2< \epsilon^2.$$

Observe that $\Phi'=\{\sqrt{w_{i,j}}\varphi_i\}_{i, j=1}^{M,N}$ is a frame whose frame bounds are the same as  those for $\mu_{\Phi,w}$. Similarly, $\Psi'=\{\sqrt{w_{i,j}}\psi_j\}_{i, j=1}^{M,N}$ is a frame whose frame bounds are the same as  those for $\mu_{\Psi,\eta}.$ Furthermore,

$$d(\Phi', \Psi')=W_2(\mu_{\Phi, w}, \mu_{\Psi, \eta})<\epsilon$$ which implies that

$$d(F(\Phi'), F(\Psi'))^2= \sum_{i, j=1}^{M,N}\|S^{-1/2}_{\Phi}(\sqrt{w_{i,j}}\varphi_i) - S^{-1/2}_{\Psi}(\sqrt{w_{i,j}}\psi_j)\|^2<  \delta^2.$$

However,
$$\sum_{i,j}\|S^{-1/2}_{\Phi}(\sqrt{w_{i,j}}\varphi_i) - S^{-1/2}_{\Psi}(\sqrt{w_{i,j}}\psi_j)\|^2 = \sum_{i,j}w_{i,j}\|S^{-1/2}_{\Phi}\varphi_i - S^{-1/2}_{\Psi}\psi_j\|^2$$

But since $w_{i,j} = \gamma(\{\varphi_i\}, \{\psi_j\})$ we have $\sum_{j}w_{i,j} = w_i$ and $\sum_{i}w_{i,j} = v_j$ we see that
$$W^2_2(F(\mu_{\Phi, w}), F(\mu_{\Psi, \eta}))=W_2^2(\mu_{\Phi^{\dag}, w}, \mu_{\Psi^{\dag}, v}) \leq \sum_{i,j}w_{i,j}\|S^{-1/2}_{\Phi}\varphi_i - S^{-1/2}_{\Psi}\psi_j\|^2.$$
\end{proof}

Let $DPF(A, B)$ denote the set of all discrete (finite)  probabilistic frames in $\RR^d$ whose lower frame bounds are less than or equal to $A$ and whose upper bounds are greater or equals to $B$. It follows from the above result that $F$ is uniformly continuous from $DPF(A, B)$ into itself when equipped with the Wasserstein metric. Consequently, we can prove the following result.

\begin{proposition}\label{well-defined-F}
Let $\mu$ be a probabilistic frame with frame bounds $A$ and $B$. Let $\mu_k:=\mu_{\Phi_{k}, w_k}$, where  $\Phi_{k}:=\Phi_{k, w_{k}}=\{\varphi_{k}\}_{k=1}^{N_k}$ and $\nu_k:=\mu_{\Psi_{k}, v_{k}}$, where $\Psi_{k}:=\Psi_{k, v_{k}}=\{\psi_{k}\}_{k=1}^{M_k}$ be two sequences of finite probabilistic frames in $\RR^d$ such that $\lim_{k\to \infty}W_{2}(\mu, \mu_{\Phi_{k}})=\lim_{k\to \infty}W_2(\mu, \mu_{\Psi_{k}})=0$. Furthermore, suppose that the frame bounds of $\mu_{\Phi_{k}}$ are between $A/2$ and $B+A/2$. Then $$\lim_{k\to \infty}F(\mu_{\Phi_{k}})=\lim_{k\to \infty}F(\mu_{\Psi_{k}}).$$
\end{proposition}

\begin{proof}
Theorem~\ref{density-dpf} ensures the existence of the finite probabilistic frames $\mu_{\Phi_{k}}$.

Let $\delta>0$ be given. By Theorem~\ref{continuityFPF} there exists $\epsilon >0$ such that for any finite probabilistic frame $\nu$ and any $k\geq 1$, $$W_2(\mu_{\Phi_{k}}, \nu)< \epsilon \implies W_2(F(\nu), F(\mu_{\Phi_{k}}))< \delta.$$

Choose $N_\epsilon>1$ such that  for all $k>N_{\epsilon}$, $W_{2}(\mu,\mu_{\Phi_{k}}) < \frac{\epsilon}{2}$ and $W_{2}(\mu,\mu_{\Psi_{k}}) < \frac{\epsilon}{2}$. Thus, for $k\geq N_\epsilon$,  $W_{2}(\mu_{\Phi_{k}},\mu_{\Psi_{k}}) < \epsilon$, which implies that for all $k\geq N_{\epsilon}$, $W_{2}(F(\mu_{\Phi_{k}}),F(\mu_{\Psi_{k}})) < \delta$. It easily follows that $\lim_{k\to \infty}F(\mu_{\Phi_{k}})=\lim_{k\to \infty}F(\mu_{\Psi_{k}})$.

\end{proof}
We can now use this proposition to extend the definition of the map $F$ to all probabilistic frames. Let $\mu$ be a probabilistic frame with bounds $0<A\leq B<\infty.$ Let $\{\mu_{\Phi_{k}}\}_{k=1}^{\infty}$ be a sequence of finite probabilistic frames with bounds between $A/2$ and $B+A/2$ such that $lim_{k\to \infty}W_2(\mu_{\Phi_{k}}, \mu)=0$. Then, $$F(\mu)=\lim_{k\to \infty}F(\mu_{\Phi_{k}})$$ is well-defined. Before proving Theorem~\ref{maintheorem} we first identify the minimizer of ~\eqref{optimum} with $F(\mu)$.

\begin{theorem}\label{optimizer1}
Let $\mu$ be a probabilistic frame on $\RR^d$ with probabilistic frame operator $S_\mu$. Then $F(\mu)$ is the unique closest probabilistic Parseval frame to $\mu$ in the $2-$Wasserstein metric, that is $F(\mu)$ is the unique solution to~\eqref{optimum}.
\end{theorem}

\begin{proof} Set $Q=\min  W_{2}(\mu, \nu)$
where $\nu $ ranges over all Parseval probabilistic frames.

Let  $\delta>0$, and  $\mu$ be a probabilistic frame wth frame bounds $A$ and $B$. By Theorem~\ref{density-dpf}, there exists a  sequence of finite probabilistic frame $\mu_{\Phi_{k}}$ with frame bounds between $\frac{A}{2}$ and $B + \frac{A}{2}$ where $\Phi_{k}:=\Phi_{k, w(k)}=\{\varphi_{k}\}_{k=1}^{N_k} \subset \RR^d$, $w(k)=\{w_n\}_{n=1}^{N_k}\subset (0, \infty)$, and $N_k\geq 2$ such that $\lim_{k\to \infty}W_2(\mu, \mu_{\Phi_{k}})=0$.

Observe that for all $k\geq 1$, $$W_{2}(\mu, F(\mu_{\Phi_{k}}))\leq W_2(\mu, F(\mu)) + W_2(F(\mu), F(\mu_{\Phi_{k}})).$$ Choose $\epsilon>0$ as in Theorem~\ref{continuityFPF} and pick $K\geq 1$ such that $W_2(\mu, \mu_{\Phi_{K}})< \epsilon.$ Thus, $W_2(F(\mu), F(\mu_{\Phi_{K}}))< \delta.$ Consequently,

$$W_{2}(\mu, F(\mu_{\Phi_{K}}))\leq W_2(\mu, F(\mu)) + W_2(F(\mu), F(\mu_{\Phi_{K}}))<  W_2(\mu, F(\mu)) + \delta.$$  Since  $F(\mu_{\Phi_{K}})$ is a Parseval frame we conclude that $F(\mu)$ minimizes~\eqref{optimum}.

We now prove that $F(\mu)$ is the unique minimizer of  ~\eqref{optimum} by considering three cases.

\noindent {\bf Case 1.} If $\mu$ is a finite frame $\Phi=\{\varphi_i\}_{i=1}^N\subset \RR^d$, it is known that $S^{-1/2}\Phi$ is the (unique) closest  Parseval frame to $\Phi$, see Theorem~\ref{uniqueness_prob_1}, and \cite[Theorem 3.1]{CasKut07}.

\noindent {\bf Case 2.} If $\mu=\mu_{\Phi, w}$, where $\Phi=\{\varphi_i\}_{i=1}^N\subset \RR^d$, and $w=\{w_i\}_{i=1}^N \subset [0, \infty)$. Then, $\mu_{\Phi^{\dag}, w}$ where $\Phi^{\dag}=S^{-1/2}\Phi$ is the unique closest  Parseval probabilistic frame to $\Phi$. Indeed, we already know that $\mu_{\Phi^{\dag}, w}$ achieves the minimum distance Proposition~\ref{dists-to-parseval}. We now prove that it is unique. We argue by contradiction and assume that there exists  another Parseval probabilistic frame $\nu$ that achieves this distance.

First, we assume  that $\nu=\mu_{\kappa, v}$ where  $\kappa =$ $\{\kappa'_i\}_{i=1}^{M}\subset \RR^d$ with weights $v=\{v_i\}_{i=1}^M\subset [0, \infty)$. Let $\gamma \in \Gamma(\mu, \nu)$ such that $$W_2(\mu, \nu)^2=\iint\|x-y\|^2d\gamma(x,y).$$  For all $i$, $j$ let $w_{i,j} = \gamma(\varphi_i,\kappa'_j)$. Let $Q = \sum_{i=1}^{N}w_i\|\varphi_i - \varphi_i^{\dag}\|^2$, where $\varphi_i^{\dag}=S^{-1/2}\varphi_i$. Since $\kappa$ also achieved this distance we clearly have $Q = \sum_{i,j}w_{i,j}\|\varphi_i - \kappa'_j\|^2$.

We now use Lemma~\ref{frame_modi}. For each $i$, we replace the vector $\varphi_i$ and its weight $w_i$ by $M$ copies of itself (i.e.,  $\varphi_i$ ) each weighted by $w_{i,j}$. Apply the same procedure to  $\Phi^{\dag}$, and to $\kappa$, except that for the latter we break each vector $\kappa_j'$ into $N$ copies of itself with weights $w_{i,j}$.  Denote by $F_1, F_2,$ and $F_3$ the three resulting frames. We note that the vectors in each of these frames can be considered to have weight $1$.

It follows  from Theorem~\ref{uniqueness_prob_1} that the finite frame $F_3=\{\sqrt{w_{i,j}}\kappa'_j\}_{i,j}$ is the (unique) closest Parseval frame to $F_1=\{\sqrt{w_{i,j}}\varphi_i\}_{i,j}$, which we also know is $F_2=\{\sqrt{w_{i,j}}\varphi_i^{\dag}\}_{i,j}$. Therefore,  $\mu_{\kappa, v}=\mu_{\Phi^{\dag}, w}$.

Next, we assume that $\nu$ is not discrete. Choose  a sequence of finite Parseval frames $\{\nu_n\}_{n}^\infty$  such that
$\lim_{n \to \infty}W_2(\nu_n, \nu)=0$.  Hence, $$Q=W_{2}(\mu, F(\mu))=W_2(\mu_{\Phi, w}, \nu)=\lim_{n \to \infty}W_2(\mu_{\Phi, w}, \nu_n).$$ We now prove that $$\lim_{n\to \infty}W_2(\nu_n, \mu_{\Phi^{\dag}, w})=0.$$

Let $\delta > 0$ and choose $N\geq 1$  such that for all $n>N$
$$W_2(\nu_n, \mu_{\Phi, w})< Q + \delta. $$

Suppose by contradiction that $\lim_{n\to \infty}W_2(\nu_n, \mu_{\Phi^{\dag}, w})> 0$. Thus, there is $\epsilon>0$ such for all $k\geq 1$, there exists $n>\max(k, N)$ such that $$W_2(\nu_n, \mu_{\Phi^{\dag}, w})>\epsilon.$$

For  $n$ given above,  let $\gamma_n \in \Gamma(\nu_n, \mu_{\Phi, w})$ be such that $$W_2^2(\nu_n, \mu_{\Phi, w})=\iint_{\RR^d}\|x-y\|^2d\gamma_n(x,y).$$
Since $\nu_n$ is a finite probabilistic frame we may assume further that  $\nu_n=\mu_{u_{n}, v}$ where $u_n= \{\psi_i\}_{i=1}^M\subset \RR^d$ and $v=\{v_i\}_{i=1}^M \subset [0, \infty)$. For the sake of simplicity in notations, we omit the dependence of both $\psi_i$ and $v_i$ on $n$.    Let $w_{n,j,k} = \gamma_{n}(\varphi_j,\psi_k)$.

Now consider the finite frames $\{u'_j\}_{j}=\{\sqrt{w_{n,j,k}}\psi_k\}_{j,k}$ and $\Phi' =$ $\{\sqrt{w_{n,j,k}}\varphi_j\}_{j,k}$.

Note that $W_2(\mu_{\Phi'}, \mu_{\Phi'^{\dag}})=Q$.  Now we consider the rows of these frames written with respect to the eigenbasis of the frame operator $S:=S_{\Phi'}$ of  $\Phi'$.

Because, $W_2(\nu_n, \mu_{\Phi^{\dag}, w})>\epsilon,$ then  $\sum_{j,k}w_{n,j,k}\|\psi_k - S^{-1/2}\varphi_j\|^2 > \epsilon$.

Using this and Lemma~\ref{contradiction-unifcont} we have the following estimates:
$$
W_2^2(\mu_{\Phi, w}, \nu_{n})\geq W_2^2(\mu_{\Phi, w}, \nu)+ \min( \tfrac{\epsilon^2}{d} \cdot M,M^2)$$ where $A$ is the lower frame bound of $\Phi$ and $M = \min(1,\sqrt{A})$.

Consequently,

$$W_2^2(\mu_{\Phi, w}, \nu_{n})-   Q^2\geq  \min( \tfrac{\epsilon^2}{d} \cdot M,M^2)>0 .$$
But,  this contradicts the fact that $Q=W_2(\mu_{\Phi, w}, \nu)=\lim_{n \to \infty}W_2(\nu_{\Phi, w},\nu_n).$  Hence, $\lim_{n\to \infty}W_2(\nu_n, \mu_{S^{-1/2}\Phi, w})= 0$, and $\nu=\mu_{\Phi^{\dag}, w}.$

\noindent {\bf Case 3:} Next, we suppose that $\mu$ is non discrete probabilistic frame with frame bounds $A,$ and $B$. Let $\{\mu_n\}_{n=1}^\infty=\{\mu_{\Phi_{n}, w(n)}\}_{n=1}^\infty$ be a sequence of finite probabilistic frames with bounds between $A/2$ and $B+A/2$ such that $\lim_{n\to \infty}W_2(\mu_n, \mu)=0$. Then $F(\mu)=\lim_{n\to \infty}F(\mu_{n})$ is such that $Q=W_{2}(F(\mu), \mu).$ Suppose there exists another Parseval frame $\nu$  such that $Q=W_2(\nu, \mu)$. Choose a sequence of finite Parseval $\{\nu_n\}_{n=1}^\infty$  such that $\lim_{n\to \infty}\nu_n=\nu.$

Observe that $Q=\lim_{n\to \infty}W_2(\mu_n, F(\mu_n))=\lim_{n\to \infty}W_2(\nu_n, \mu_n)$.
Write  $\Phi_n  = \{\varphi_{n,j}\}_{j=1}^{M}$ and $w(n)=\{w_{j}\}_{j=1}^M$, where for simplicity we omit the dependence of $M$ on $n$. Similarly,  $\{\nu_n\}_{n=1}^{\infty}= \{\psi_{n,j}\}_{j=i}^{M'}$  with weights $v(n)= \{v_{j}\}_{j=1}^{M'}$.

 Let $\gamma_n\in \Gamma(\mu_n, \nu_n)$  be such that $$W_2^2(\mu_{n},\nu_n) =\iint \|x-y\|^2d\gamma_n(x,y).$$  Set $$w_{j,k} = \gamma_n(\varphi_{n,j},\psi_{n,k})$$
We know that
\begin{align*}
W_2^2(\mu_n, F(\mu_n)) &= \sum_{j= 1}^{M}w_j\|\varphi_{n,j} - \varphi^{\dag}_{n,j}\|^2 = \sum_{j,k}w_{j,k}\|\varphi_{n,j} - \varphi^{\dag}_{n,j}\|^2 \\
&=\sum_{j,k}\|\sqrt{w_{j,k}}\varphi_{n,j} - \sqrt{w_{j,k}}\varphi^{\dag}_{n,j}\|^2
\end{align*}

We also know that
$$
W_2^2(\mu_{n},\nu_n)  =  \sum_{j,k}w_{j,k}\|\varphi_{n,j} - \psi_{n,k}\|^2 =\sum_{j,k}\|\sqrt{w_{j,k}}\varphi_{n,j} - \sqrt{w_{j,k}}\psi_{n,k}\|^2 $$

Suppose that $\lim_{n\to \infty} W_2(F(\mu_n), \nu_n)>0$. Thus, there exists $\epsilon>0$ and and integer $n>1$ such that $W_2(F(\mu_n),\nu_n) > \epsilon$. Consequently,
$$\epsilon< \sum_{j,k}w_{j,k}\|\varphi^{\dag}_{n,j} - \psi_{n,k}\|^2 = \sum_{j,k}\|\sqrt{w_{j,k}}\varphi^{\dag}_{n,j} - \sqrt{w_{j,k}}\psi_{n,k}\|^2 $$

Hence  $$d(\Phi'^{\dag }_n, \Psi_n')>\epsilon$$ where $\Psi'_n = \{\sqrt{w_{j,k}}\psi_{n,k}\}$.

By the same argument as in Lemma~\ref{contradiction-unifcont}  we conclude that  $W_2^2(\mu_{n},\nu_n)- W_2^2(\mu_n, F(\mu_n))\geq min(M \frac{\epsilon^2}{d},M^2).$ where $M = \min(1,\sqrt{\frac{A}{2}})$

This contradicts the fact that Since $\lim_{n\to \infty}W_2(\mu_n,\nu_n) = Q = \lim_{n\to \infty}W_2(\mu_n,F(\mu_n))$. Thus $\lim_{n\to \infty} W_2(F(\mu_n), \nu_n)=0$ and so $F(\mu)=\nu$.

\end{proof}

By Proposition~\ref{well-defined-F} it follows that given a probabilistic frame $\mu$ and any sequence $\Phi_{k}:=\Phi_{k, w_{k}}=\{\varphi_{k}\}_{k=1}^{N_k}$ of finite probabilistic frames in $\RR^d$ such that $\lim_{k\to \infty} W_{2}(\mu, \mu_{\Phi_{k}})=0$, then $F(\mu)=\lim_{k\to \infty}F(\mu_{\Phi_{k}}).$ Furthermore, it is proved in \cite{WCKO17} that if $\{\mu_n\}_{n\geq 1} \subset \mathcal{P}_2$ converges in the Wassertein metric to $\mu \in \mathcal{P}_2$, then $$\|S_{\mu}-S_{\mu_{n}}\|\leq CW_2(\mu_{n}, \mu).$$

All that is needed to prove Theorem~\ref{maintheorem} is to show that $F(\mu)=\mu^{\dag}$.

\begin{proof}[Proof of Theorem~\ref{maintheorem}]
Let $\mu$ be a probabilistic frame with bounds $A, B$. Let $0< \epsilon<A/2$ and choose a compactly supported probabilistic frame $\nu_{\epsilon}$ as in Lemma~\ref{prop-2}. In particular $\nu_{\epsilon}$ is supported on $B(0,R_{\epsilon})$ with frame bounds between $A/2$ and $B+A/2$, where $R_{\epsilon}>0$ is such that $$\int_{\RR^d\setminus B(0,R_{\epsilon})}\|x\|^2dx< \epsilon/3.$$  

Choose a finite probabilistic frame $\mu_{\epsilon}$ with bounds between $\frac{A}{2}$ and $B + \frac{A}{2}$ such that $W_{2}(\mu_{\epsilon}, \nu_{\epsilon}) < \frac{\epsilon}{3}$. By taking  a sequence $\{\epsilon_n\}_{n=1}^{\infty}\subset [0, \infty)$  with $\lim_{n\to \infty}\epsilon_n=0$, we can pick $\{\mu_n\}_{n\geq 1}:=\{\mu_{\epsilon_n}\}_{n\geq 1}$ such that  $\lim_{n\to \infty}W_{2}(\mu_{n}, \mu)=0$. Consequently,  $\lim_{n\to \infty}S_{\mu_{n}}=S_{\mu}$, and $\lim_{n\to \infty}S^{-1/2}_{\mu_{n}}=S^{-1/2}_{\mu}$ in the operator norm. 

We recall that $\lim_{n\to \infty}W_{2}(\mu_{n}, \mu)=0$ is equivalent to

$$
\lim_{n\to \infty}\int f\, d\mu_{n}(x) =\int f\, d\mu(x) \\
$$ for all continuous function $f$ such that $|f(x)|\leq C(1+\|x-x_0\|^2)$ for some $x_0\in \RR^d$ \cite[Theorem  6.9]{Villani2009}

We know that $\lim_{n\to \infty} F(\mu_n)=\lim_{n\to \infty} \mu_{n}^{\dag}=F(\mu)$ in the Wasserstein metric. We would like to show that $\lim_{n\to \infty} F(\mu_n)=\lim_{n\to \infty}\mu_{n}^{\dag}=\mu^{\dag}$.

We  show that for all continuous function $f$ such that $|f(x)|\leq C(1+\|x-x_0\|^2)$ for some $x_0\in \RR^d$ $$\lim_{n\to \infty}\int f\, d\mu_{n}^{\dag}(x) =\int f\, d\mu^{\dag}(x).$$

\begin{align*}
|\int f\, d\mu_{n}^{\dag}(x) -\int f\, d\mu^{\dag}(x)|& = |\int f(S^{-1/2}_{\mu_{n}} x)\, d\mu_{n}(x) -\int f(S^{-1/2}_{\mu} x)\, d\mu(x)|\\
&\leq \int |f(S^{-1/2}_{\mu_{n}} x) - f(S^{-1/2}_{\mu} x)|\, d\mu_{n}(x) +\\
&|\int f(S^{-1/2}_{\mu} x)\, d\mu_{n}(x)  -\int f(S^{-1/2}_{\mu} x)\, d\mu(x)|\\
\end{align*}

Let $f$ be continuous with $|f(x)|\leq C(1+\|x-x_0\|^2)$ for some $x_0\in \RR^d$. Then, $f(S^{-1/2}_\mu)$ is continuous and satisfies  $$|f(S^{-1/2}_\mu x)|\leq C(1+\|x_0-S^{-1/2}_\mu x\|^2)\leq C(1+\|S^{-1/2}_\mu\|^2\|x-S^{1/2}_{\mu}x_0\|^2)\leq C' (1+\|x-S^{1/2}_{\mu}x_0\|^2)).$$ Consequently, we can find $N_1$ such that for all $n\geq N_1$, $$|\int f(S^{-1/2}_{\mu} x)\, d\mu_{n}(x)  -\int f(S^{-1/2}_{\mu} x)\, d\mu(x)|< \epsilon/3.$$

Since $f$ is continuous, there exists  $\delta>0$ such that for all $x, y \in B(0,R')$, $\|x-y\|< \delta $ implies that $|f(x)-f(y)|< \epsilon/3$, where $R'>0$ is chosen  so as to guarantee that for large $n$, and $x\in B(0,R)$, $S^{-1/2}_{\mu_{n}}x, S^{-1/2}_{\mu}x \in B(0, R)$. Since,  $\lim_{n\to \infty}S^{-1/2}_{\mu_{n}}=S^{-1/2}_{\mu}$ , there exists $N_2$ such that for all $n\geq N_2$,
$$\|S^{-1/2}_{\mu_{n}}x-S^{-1/2}_{\mu}x\|\leq \|S^{-1/2}_{\mu_{n}}-S^{-1/2}_{\mu}\| \|x\|\leq R \|S^{-1/2}_{\mu_{n}}-S^{-1/2}_{\mu}\|< \delta.$$

Therefore, for $n\geq N_2$, $ |f(S^{-1/2}_{\mu_{n}} x) - f(S^{-1/2}_{\mu} x)|< \epsilon/3$ for all $x\in B(0, R)$. Consequently,
\begin{align*}
\int |f(S^{-1/2}_{\mu_{n}} x) - f(S^{-1/2}_{\mu} x)|\, d\mu_{n}(x) &= \int_{B(0,R)} |f(S^{-1/2}_{\mu_{n}} x) - f(S^{-1/2}_{\mu} x)|\, d\mu_{n}(x)\\&+\int_{\RR^d\setminus B(0,R)} |f(S^{-1/2}_{\mu_{n}} x) - f(S^{-1/2}_{\mu} x)|\, d\mu_{n}(x) \\
&<  \epsilon/3+ \int_{\RR^d\setminus B(0,R)} |f(S^{-1/2}_{\mu_{n}} x) - f(S^{-1/2}_{\mu} x)|\, d\mu_{n}(x)\\
&< \epsilon/3 + M \int_{\RR^d\setminus B(0,R)} \|x\|^2\, d\mu_{n}(x)\\
&<2\epsilon/3
\end{align*}
where $M>0$ is a constant that depends only on $f$, and $\mu$.

It follows that for all $n\geq \max(N_1, N_2),$  we have
$$|\int f\, d\mu_{n}^{\dag}(x) -\int f\, d\mu^{\dag}(x)|< \epsilon$$ which implies that $\lim_{n\to \infty}\int f\, d\mu_{n}^{\dag}(x) =\int f\, d\mu^{\dag}(x).$

\end{proof}

\section*{Acknowledgment}
Both authors were partially supported by ARO grant W911NF1610008. K.~A.~Okoudjou  was also partially supported by a grant from the Simons Foundation $\# 319197$. This material is based upon work supported by the National Science Foundation under Grant No.~DMS-1440140 while K.~A.~Okoudjou was in residence at the Mathematical Sciences Research Institute in Berkeley, California, during the Spring 2017 semester.

\bibliographystyle{amsplain}
\bibliography{PFW2_bib}

\end{document}